\documentclass{amsart}
\usepackage{mathtools}
\usepackage{float}
\usepackage{tikz-cd}
\newtheorem{theorem}{Theorem}[section]
\newtheorem{lemma}[theorem]{Lemma}

\theoremstyle{definition}
\newtheorem{definition}[theorem]{Definition}

\newtheorem{example}[theorem]{Example}
\newtheorem{claim}[theorem]{Claim}
\newtheorem{cor}[theorem]{Corollary}
\newtheorem{conj}[theorem]{Conjecture}
\newtheorem{theoremx}{Theorem}
\newtheorem{corx}[theoremx]{Corollary}

\theoremstyle{remark}
\newtheorem{remark}[theorem]{Remark}

\numberwithin{equation}{section}

\newcommand{\defeq}{\vcentcolon=}
\newcommand{\cl}{{\rm cl}}
\newcommand{\scl}{{\rm scl}}

\begin{document}
	\title{Spectral Gap of scl in Free Products}
	\author{Lvzhou Chen}
	\address{Department of Mathematics, University of Chicago, Chicago, Illinois, 60637}
	\email{lzchen@math.uchicago.edu}
	
	
	\date{Nov. 28, 2016. Revised on July 18, 2017.}
	
	
	\begin{abstract}
		Let $G=*_\lambda G_\lambda$ be a free product of torsion-free groups, and let $g\in[G,G]$ be any element not conjugate into a $G_\lambda$. Then $\scl_G(g)\ge1/2$. This generalizes, and gives a new proof of a theorem of Duncan--Howie \cite{DH}.
	\end{abstract}
	\maketitle

	\section{Introduction}
	For any group $G$, let $[G,G]$ denote its commutator subgroup. For any $g\in[G,G]$, the \textit{commutator length} $\cl(g)$ is the minimal number $n$ such that $$g=[a_1,b_1][a_2,b_2]\cdots[a_n,b_n]$$ for some $a_i,b_i\in G$, and the \textit{stable commutator length} $\scl(g)$ is the limit $$\lim_{n\to\infty}\frac{\cl(g^n)}{n}.$$
	
	The \textit{spectrum} of $\scl$ is the set of values of $\scl(g)$ as $g$ runs over elements of $[G,G]$.
	
	\subsection{Main results}In this paper, our main result is
	\begin{theoremx}\label{fprod}
		Let $G=*_\lambda G_\lambda$ be a free product of torsion-free groups, and suppose $g\in [G,G]$ is not conjugate into any $G_\lambda$, then $$\scl_G(g)\ge1/2.$$
	\end{theoremx}
	This statement must be modified when the factors have torsions. For instance, we have a lower bound $1/2-1/N$ if every nontrivial element in each factor group has order at least $N\ge2$; for details, see Theorem \ref{refine}. The same statement and proof are still valid if the assumption $g\in [G,G]$ is weakened to $g^n\in[G,G]$ for some $n\ge1$ since we use the geometric interpretation (See Section \ref{sec2}).
	
	A special case is when each $G_\lambda=\mathbb{Z}$. In this case $G$ is free, and no $g\in[G,G]\backslash\{id\}$ is conjugate into a factor. Thus we obtain a new proof of
	\begin{corx}\label{free}
		If $F$ is free, and $g\in[F,F]\backslash\{id\}$, then $$\scl(g)\ge1/2.$$
	\end{corx}
	Duncan--Howie \cite{DH} proved Theorem \ref{fprod} when $G_\lambda$ are locally indicable. Our proof is new even in that case.
	
	A group $G$ with the property that either $\scl(g)=0$ or $\scl(g)\ge C$ for some $C=C(G)>0$ for all $g\in[G,G]$ is said to have a \emph{spectral gap} $C$ for $\scl$. Residually free groups have spectral gap $1/2$ (\cite{DSCL} Corollary 4.113 using Duncan--Howie's result); $\delta$-hyperbolic groups have a spectral gap that can be estimated by the number of generators and $\delta$ (Calegari--Fujiwara \cite{CF}); finite index subgroups of mapping class groups also have a spectral gap (Bestvina--Bromberg--Fujiwara \cite{BBF}); Baumslag--Solitar groups have a spectral gap $1/12$ (Clay--Forester--Louwsma \cite{CFL}); Right angled Artin groups have a spectral gap $1/24$ (Fern\'os--Forester--Tao \cite{FFT}).
	
	Our results imply that $*G_i$ has a spectral gap $\min\{C,1/2-1/N\}$ ($N\ge3$) if all $G_i$ have spectral gap $C$ and contains no ($N-1$)-torsion. Without the assumption on torsions, the spectral gap $\min\{C,1/12\}$ has been obtained in \cite{CFL}, Theorem 6.9.
	
	In fact we give two logically independent proofs of Corollary \ref{free}.
	
	Ivanov--Klyachko \cite{IK} recently independently obtained Theorem \ref{fprod} together with other estimates related to Theorem \ref{refine} in terms of commutator length. Their argument uses a different language (diagrams) but the idea behind is similar, especially in the case of free groups. Corollary \ref{clversion} provides the connection.
	
	\subsection{Contents of paper}Section \ref{sec2} introduces the geometric language of fatgraphs, used to study $\scl$ in free groups. We give a new proof of Corollary \ref{free} and discuss potential generalizations to integral chains.
	
	Section \ref{sec3} introduces some techniques to study surface maps into a wedge of spaces. We use these techniques to prove Theorem \ref{refine} which allows the factors to have torsion, then we deduce Theorem \ref{fprod}.
	\subsection{Acknowledgments}
	
	The author would like to thank his advisor Danny \linebreak Calegari, Sergei Ivanov, Anton Klyachko, Alden Walker and the anonymous referee for reading earlier versions of this paper and giving great suggestions.
	
	\section{Geometric Definition of scl}\label{sec2}
	Our arguments are geometric, and depend on an interpretation of $\scl$ in terms of maps of surfaces.
	
	Let $G$ be a group. Let $X$ be a $K(G,1)$. Each conjugacy class $g\in[G,G]$ corresponds to a free homotopy class of loop $\gamma:S^1\to X$.
	
	An \emph{admissible} surface for $g$ is a compact, oriented surface $R$ without disk or sphere components, together with a free homotopy class of map $f:R\to X$ for which the following diagram commutes
	\begin{center}
		\begin{tikzcd}
			\partial R \arrow{r}{\partial f} \arrow[hookrightarrow]{d}
			&S^1\arrow{d}{\gamma}\\
			R\arrow{r}{f}
			& X 
		\end{tikzcd}
	\end{center}
	and $\partial f$ is a positively oriented (possibly disconnected) covering (of degree $n(R)$).
	
	These are sometimes called \emph{monotone} admissible surfaces (\cite{DSCL}, Definition 2.12).
	\begin{lemma}[\cite{DSCL}, Proposition 2.10 and 2.13]\label{geomdef}
		$$\scl(g)=\inf\frac{-\chi(R)}{2n(R)}$$ over all admissible surfaces for $g$.
	\end{lemma}
	This also defines $\scl(g)$ for $g\in G$ with $g^n\in [G,G]$ for some $n\ge1$.
	
	The following corollary is well-known to people studying scl, we include it for readers interested in commutator length.
	\begin{cor}\label{clscl}
		Let $g_i$ be conjugates of $g$. Unless $m=1$ and $g_1^{n_1}=id$, we have
		$$\cl(g_1^{n_1}\ldots g_m^{n_m})\ge \scl(g)\left|\sum_{i=1}^{m} n_i\right|-\frac{m}{2}+1.$$
	\end{cor}
	\begin{proof}
		If $g_1^{n_1}\ldots g_m^{n_m}=id$, then the inequality does not hold only when $m=1$. From now on, assume $g_1^{n_1}\ldots g_m^{n_m}\neq id$ and can be written as a product of $k$ commutators, then we obtain a surface $R$ of genus $k$ with $m$ boundary components and a map $f:R\to X$ such that the boundary components wrap $n_i$ times respectively around a loop $\gamma$ representing the conjugacy class $g$. Then $-\chi(R)=2k+m-2$ and $R$ is not a disk by assumption. If all $n_i$ have the same sign, the inequality follows from Lemma \ref{geomdef}; for the general case, apply Proposition 2.10 in \cite{DSCL} instead.
	\end{proof}

	If $G$ is free, we can take $X$ to be a wedge of circles. In this case, any admissible $S$ can be compressed (reducing $-\chi(S)$ without changing $n(R)$) until it is represented by a \emph{fatgraph} (see \cite{Cul}). See \cite{SSLP} or \cite{Fat} for an introduction to fatgraphs.
	
	Informally, a fatgraph is a graph $Y$ with a cyclic ordering of edges at each vertex, which lets $Y$ embed canonically as the spine of a compact oriented surface $S(Y)$ (the \emph{fattening}) which deformation retracts to $Y$. $\partial S(Y)$ has an induced simplicial structure.
	
	If $Y$ comes with a simplicial map $f:Y\to X$, then we get a surface map $\bar{f}:S(Y)\to X$, which is simplicial on $\partial S(Y)$, by pre-composing with the deformation retraction $S(Y)\to Y$. We decorate the oriented edges of $\partial S(Y)$ by generators of $F$ to indicate where they are mapped in $X$ by $\bar{f}$. See Figure \ref{fat} for an example.
	
	Any cyclically reduced $g$ in a free group $F$ is represented by a simplicial immersion $\phi: C_{|g|}\to X$, where $|g|$ is the word length, and $C_{|g|}$ is the simplicial oriented $S^1$ with $|g|$ vertices. For an admissible fatgraph $Y$ (with a simplicial map $f:Y\to X$), the oriented covering map $\partial \bar{f}:\partial S(Y)\to C_{|g|}$ can be taken to be simplicial. Every admissible surface (up to compression and homotopy) can be put in this form.
	
	\begin{figure}
		\begin{center}
			\resizebox{324pt}{148pt}{\includegraphics{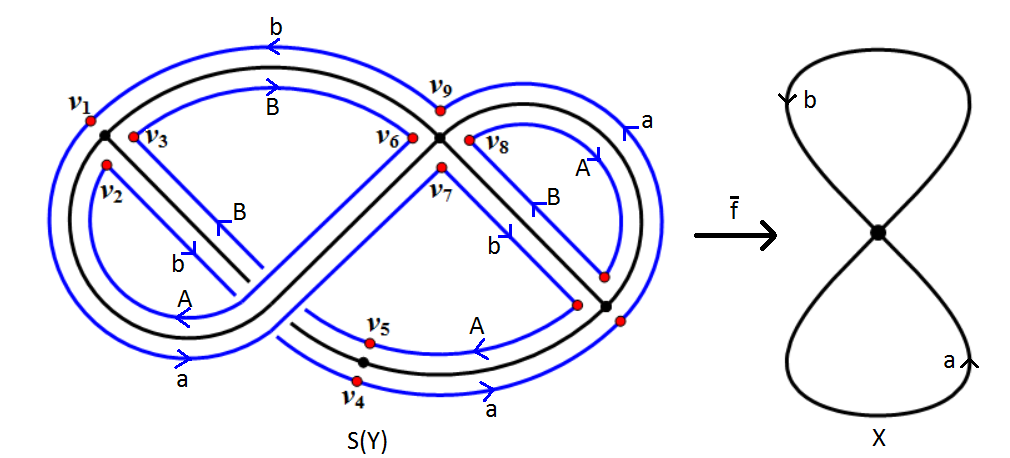}}
			\caption{Fattening and Decoration} \label{fat}
		\end{center}
	\end{figure}
	
	Now we prove
	\setcounter{theoremx}{1}
	\begin{corx}
		If $F$ is free, and $g\in[F,F]\backslash\{id\}$, then $$\scl(g)\ge1/2.$$
	\end{corx}
	\begin{proof}
	Take any fatgraph $Y$ with fattening $S=S(Y)$ admissible of degree $n$ for $g$. Label the vertices of $C_{|g|}$ cyclically as $1,2,\ldots,|g|$. Pull back the labels to $\partial S$ via the covering map $\partial\bar{f}$, then each edge on $\partial S$ also gets labeled as $(i,i+1)$ for $i<|g|$ or $(|g|,1)$.
		
	Two (distinct) edges of $\partial S$ are \emph{paired} if they are mapped to the same edge of $Y$ under the deformation retraction. Two (distinct) vertices of $\partial S$ are \emph{paired} if they are end points of a pair of paired edges that correspond to some edge $e$ in $Y$ and these two vertices correspond to the same end point of $e$. In Figure \ref{fat}, $v_4$ and $v_5$ are paired vertices; $v_1$,$v_2$ and $v_3$ are mutually paired; $v_6$ is paired with $v_7$ and $v_9$ but not with $v_8$.
	
	\begin{claim}\label{noloop}
		Paired vertices have distinct labels.
	\end{claim}
	\begin{proof}
		Suppose not, then we will have two paired edges mapped to two consecutive edges of $C_{|g|}$ under $\partial\bar{f}$. But paired edges are decorated by inverse letters, so the cyclic word decorating $\partial S$ is not cyclically reduced, contrary to the assumption.
	\end{proof}
				
	Now construct a directed graph $G$ (possibly with multi-edge) as follows. The vertex set is $\{1,2,\ldots,|g|\}$. Whenever we have a pair of paired edges on $\partial S$ labeled as $(i,i+1)$ and $(j,j+1)$ respectively, add a directed edge from $i+1$ to $j$ and another from $j+1$ to $i$. We say a directed edge from $i$ to $j$ is \emph{descending} if $i>j$. See Figure \ref{Gexample} for an example. The graph $G$ resembles the turn graph introduced by Brady--Clay--Forester \cite{BCF} to compute scl in free groups.
	
	\begin{figure}[H]
		\begin{center}
			\resizebox{331pt}{130pt}{\includegraphics{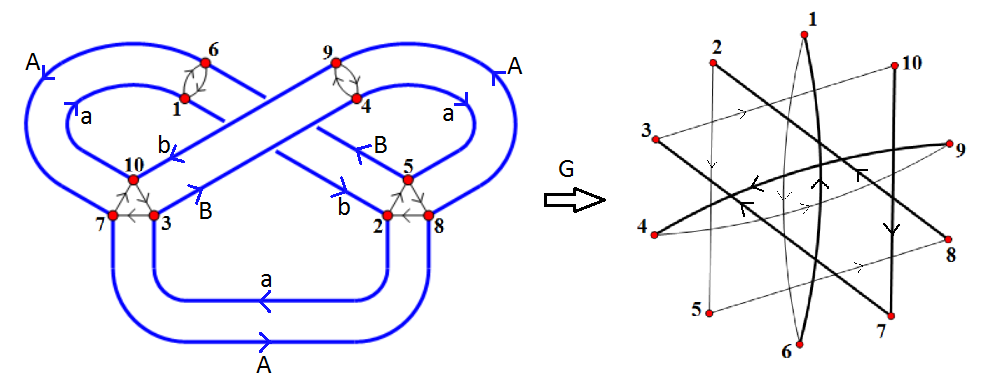}}
			\caption{The pull back label on a fatgraph that is admissible for $g=abaBaBAAAb\in F_2$, and the corresponding graph $G$ with descending edges thickened.} \label{Gexample}
		\end{center}
	\end{figure}
	
	For each vertex $v$ of $Y$, let $d(v)$ denote the valence.
	\begin{claim}
		$$\sum d(v)=n|g|\text{, where the sum is over all vertices of }Y.$$
	\end{claim}
	\begin{proof}
		For each vertex $v$ of $Y$ with valence $d(v)$, there are exactly $d(v)$ vertices on $\partial S$ that deformation retract to $v$. Thus $\sum d(v)$ is the number of vertices of $\partial S$, which equals the number of edges on $\partial S$, which is $n|g|$.
	\end{proof}
	
	\begin{claim}
		$$\#(\text{vertices in }Y)\le n\left(\frac{1}{2}|g|-1\right).$$
	\end{claim}
	\begin{proof}
		In the proof above, we see that there are exactly $d(v)$ vertices on $\partial S$ that deformation retract to $v$. These vertices contribute to exactly $d(v)$ directed edges in $G$ which form a directed cycle. This gives a decomposition of $G$ into cycles as $v$ ranges over all vertices of $Y$. Note that each directed cycle in $G$ must contain a descending edge since there is no self loop according to Claim \ref{noloop}. Therefore, the number of vertices in $Y$ is no more than the number of descending edges in $G$.
				
		On the other hand, for each pair of paired edges on $\partial S$ labeled as $(i,i+1)$ and $(j,j+1)$ respectively, if neither of $i,j$ is $|g|$, then exactly one of the two directed edges in $G$ contributed by this pair is descending; if $i=|g|$ or $j=|g|$, then neither of the two directed edges is descending. Thus the number of descending edges in $G$ is $n(|g|/2-1)$.
	\end{proof}
		
	Combining the results above, we have
	\begin{eqnarray*}
		-\chi(S)=-\chi(Y)=\sum\frac{d(v)-2}{2}&=&\frac{1}{2}\sum d(v)-\#(\text{vertices of }Y)\\
		&\ge&\frac{1}{2}n|g|-n\left(\frac{1}{2}|g|-1\right)\\
		&=&n
	\end{eqnarray*} 
	for all admissible fatgraphs, which implies $\scl(g)\ge 1/2$.
	\end{proof}
	
	$\scl$ is extended to integral chains (formal sums of elements) in \cite{DSCL}.
	\begin{conj}[Calegari]\label{genspec}
		$$\scl(c)\ge\frac{1}{2}$$ for any integral chain $c$ in a free group, unless an admissible surface of $c$ is annuli (in which case $\scl(c)=0$).
	\end{conj}
	\begin{remark}
		Using the argument above, any ordering of the vertices on $C_{|g|}$ gives a lower bound on $\scl(g)$, which also works for chains. However, few orderings provide the correct lower bound $1/2$, and it seems difficult to show such good orderings exist for general chains. Computer experiments give evidence for Conjecture \ref{genspec}.
	\end{remark}
	\begin{remark}
		Duncan--Howie's proof depends on the existence of a left-ordering on torsion-free one-relator quotients of the free groups. There is no analogy of their argument for integral chains. Thus one motivation of our work is to find a new proof of their result which does not depend on orderability.
	\end{remark}
	\begin{remark}
		Spectral gap $1/2$ is often useful to certify \emph{extremal} surfaces (those admissible surfaces realizing the infimum in the geometric interpretation). For example, Corollary \ref{free} implies that the once-punctured torus bounding $[x,y]$ is extremal when $x,y\in F$ do not commute. Similarly, the special case $c=x^{-1}+y^{-1}+xy$ in Conjecture \ref{genspec} is asking whether the thrice-punctured sphere bounding $c$ is extremal when $x,y$ do not commute, which is still open.
	\end{remark}
	
	\section{Free product case}\label{sec3}
	In this section we prove Theorem \ref{fprod}. This will follow by induction and a finiteness argument from:
	\begin{theorem}\label{refine}
		Let $G=A*B$ and $g=a_1b_1\cdots a_Lb_L$ with $a_i\in A\backslash\{id\}$, $b_i\in B\backslash\{id\}$ and $L\ge1$ such that $g\in[G,G]$. Let $N\ge2$ be the minimal order of $a_i$ and $b_i$, then $$\scl_G(g)\ge1/2-1/N.$$
	\end{theorem}
	The formalism of fatgraphs is inadequate when factors are not free; thus we introduce a new formalism following \cite{DSSS}.
	
	If $G=A*B$, then we can build a $K(G,1)$ by taking $X\defeq K(A,1)\vee K(B,1)$ to be a wedge. Then an admissible surface $R\to X$ decomposes into subsurfaces $R_A\to K(A,1)$ and $R_B\to K(B,1)$. $R_A$ and $R_B$ are surfaces with \emph{corners}, which each contributes $1/4$ to $-\chi$.
	
	Calegari \cite{DSSS} shows how to compute $\scl$ in certain free products by a pair of linear programming problems, one for each of $A$ and $B$.
	
	When $A,B$ are abelian (the case Calegari considers), the contribution of $R_A$ to $\scl$ comes from a \emph{linear} term, together with a nonlinear contribution from \emph{disk components}, i.e. components of $R_A$ which are homeomorphic to $D^2$.
	
	\begin{figure}
		\begin{center}
			\resizebox{319pt}{129pt}{\includegraphics{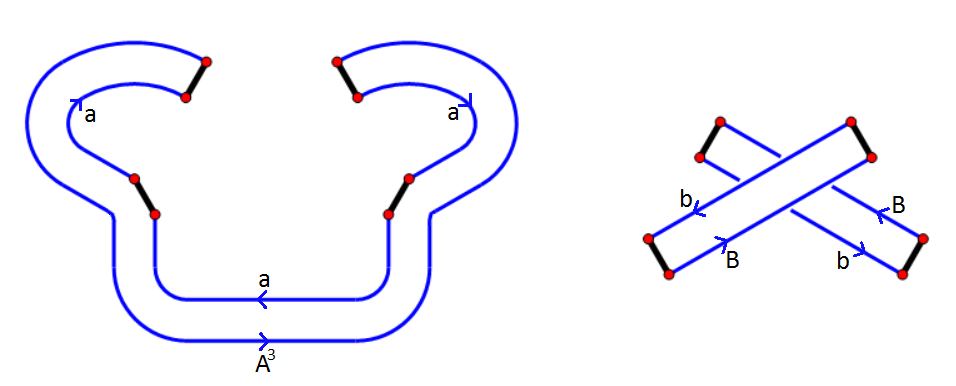}}
			\caption{$R_A$ and $R_B$ with inessential arcs thickened and corners represented by dots.} \label{RARB}
		\end{center}
	\end{figure}
	
	\begin{example}
		Let $R$ be the fatgraph admissible for $g=abaBaBA^3b\in \mathbb{Z}*\mathbb{Z}$ in Figure \ref{Gexample}. It decomposes into $R_A$ and $R_B$ as in Figure \ref{RARB}. $R_A$ contributes $$1/4 \#\text{corners}-\#\text{disks}=2-1=1$$ to $-\chi(R)$, and similarly the contribution of $R_B$ is $0$. In general, to minimize $-\chi(R)/2n$ is to maximize the number of disk components --- a nonlinear problem.
	\end{example}
	
	For arbitrary $A,B$ we obtain a lower bound on $\scl$ by ignoring the (positive) contribution to $-\chi$ of non-disk components of $R_A$ and $R_B$. Equality holds (by a covering argument, see \cite{DSSS} or \cite{My}) when $\scl$ vanishes on both $A$ and $B$. Formally, fix $G$ and $g$ as in Theorem \ref{refine}.
	\begin{definition}
		Let $W$ be a vector space formally spanned by the set of ordered pairs $(i,j)$, $1\le i,j\le L$. Let $$V\defeq\left\{\sum_{1\le i,j\le L}x_{ij}(i,j) \text{ such that }x_{ij}\ge0,\sum_i x_{ij}=1, \sum_j x_{ij}=1\right\}\subset W,$$
		$$\mathcal{D}_A\defeq\left\{\sum_{j=1}^{k}(i_j,i_{j+1}) \text{ such that }i_{k+1}=i_1,\prod_{j=1}^k a_{i_j}=1\in A,k>0\right\}\subset W.$$
		Each element in $\mathcal{D}_A$ is called a \emph{disk vector} in $A$. For any $v\in V$, define $$\kappa_A(v)\defeq\sup\left\{\sum t_i\left|\ v=\sum t_i d_i+\sum x'_{ij}(i,j),t_i\ge0,d_i\in\mathcal{D}_A, x'_{ij}\ge0\right.\right\}.$$
		Define $\mathcal{D}_B$ and $\kappa_B$ similarly. Finally define $\phi:V\to V$ to be the affine map given by $\phi(i,j)=(j-1,i)$ (replace $j-1$ by $L$ if $j=1$).
	\end{definition}

	Up to compression of $R$, the boundary $\partial R_A$ alternates between arcs mapped to one of $a_i$ and \emph{inessential} arcs (mapped to the wedge point). Encode $R_A$ as $v_A=\sum x_{ij}(i,j)\in V$, where $x_{ij}$ is the number, divided by $n(R)$, of inessential arcs on $\partial R_A$ that go from $a_i$ to $a_j$. Disk components contribute to disk vectors. Such an encoding cannot reconstruct $R_A$ but is enough to bound from below the contribution of $R_A$ to $-\chi(R)$, and $\kappa_A(v_A)$ is the (normalized) maximal number of disk components $R_A$ can have. 
	
	Note that $R_A$ and $R_B$ can be glued up along inessential arcs, such that $a_i$ should be followed by $b_i$ and $b_{j-1}$ should be followed by $a_j$. Thus the vectors $v_A$ and $v_B$ corresponding to $R_A$ and $R_B$ satisfy $v_B=\phi(v_A)$. Based on these, $\scl_G$ can be estimated using the following lemma.
	\begin{lemma}\label{compscl}
		Under the notation above, $$2\cdot\scl_G(g)\ge L-\sup_{v_A\in V}\left\{\kappa_A(v_A)+\kappa_B(\phi(v_A))\right\}.$$
		Equality holds if $\scl_A$ and $\scl_B$ are identically zero.
	\end{lemma}
	This is Corollary 4.17 in \cite{My} in the case $G_1=A$, $G_2=B$ and $z=g$ since $(v_A,v_B)\in Y_l$ means $v_B=\phi(v_A)$ in our notation and we have $|v_A|=|v_B|=L$.
	
	Now we are ready to prove Theorem \ref{refine}.
	\begin{proof}[Proof of Theorem \ref{refine}]
		By Lemma \ref{compscl}, it suffices to show $$\kappa_A(v_A)+\kappa_B(\phi(v_A))\le L-1+2/N$$ for any $v_A=\sum x_{ij}(i,j)\in V$.
		
		\begin{claim}
			For all $v=\sum x_{ij}(i,j)\in V$ we have $$
			\kappa_A(v)\le\frac{1}{N}\sum_{i\ge j} x_{ij}+\left(1-\frac{1}{N}\right)\sum_{i<j} x_{ij},\text{ and similarly for }\kappa_B(v).$$
		\end{claim}
		\begin{proof}
			Suppose $v=\sum t_i d_i+\sum x'_{ij}(i,j)$ with $t_i\ge0$, $d_i\in \mathcal{D}_A$ and $x'_{ij}\ge0$. It suffices to show that each $d_i$ contributes at least $1$ to the right hand side of the inequality.
			
			For any disk vector $d=\sum_{j=1}^{k}(i_j,i_{j+1})$, if all $i_j$ are equal, then $k\ge N$, and thus the contribution of $d$ to the right hand side is $k/N\ge1$.
			
			Suppose there are at least two distinct $i_j$'s, then there exist $j_1$ and $j_2$ such that $i_{j_1}>i_{j_1+1}$ and $i_{j_2}<i_{j_2+1}$, thus the contribution of $d$ to the right hand side is at least $1/N+(1-1/N)=1$.
		\end{proof}
		
		Geometrically, each $v\in V$ can be thought as an $L\times L$ doubly stochastic matrix. Since each row sums to $1$, the estimate in the claim above is equivalent to saying that $\kappa_A(v)-L/N$ and $\kappa_B(v)-L/N$ are at most $(1-2/N)$ times the sum of entries in the strictly upper triangular region $U$. The pull back $\phi^{-1}(U)$ and $U$ together form $L-1$ columns (Figure \ref{matrix} illustrates the case $L=5$), in which the entries sum to $L-1$. Combining these, we get the desired inequality.
		
		Formally, by the claim above, for any $v_A=\sum x_{ij}(i,j)\in V$, we have
		\begin{eqnarray*}
			\kappa_A(v_A)+\kappa_B(\phi(v_A))&\le&\frac{2L}{N}+\left(1-\frac{2}{N}\right)\sum_{i<j}x_{ij}+\left(1-\frac{2}{N}\right)\sum_{1\le j-1<i} x_{ij}\\
			&=&\frac{2L}{N}+\left(1-\frac{2}{N}\right)\left[\sum_{i<j}x_{ij}+\sum_{2\le j\le i} x_{ij}\right]\\
			&=&\frac{2L}{N}+\left(1-\frac{2}{N}\right)\sum_{2\le j\le L}x_{ij}\\
			&=&\frac{2L}{N}+\left(1-\frac{2}{N}\right)(L-1)\\
			&=&L-1+\frac{2}{N}
		\end{eqnarray*}
		as desired.
	\end{proof}
	\begin{figure}
		\begin{center}
			\resizebox{245pt}{110pt}{\includegraphics{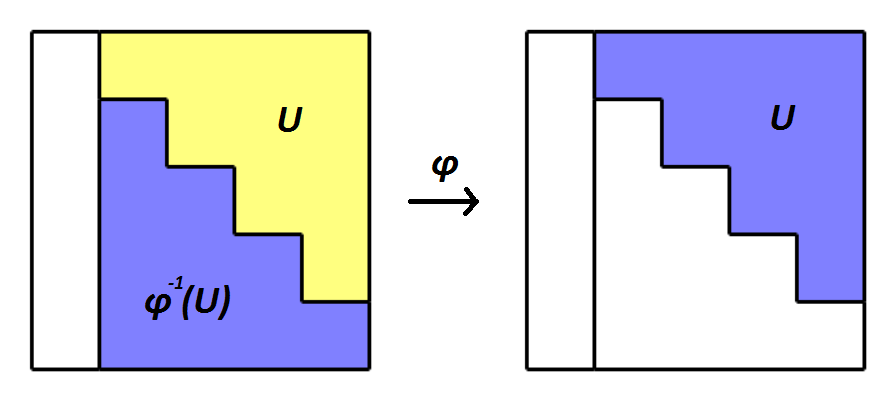}}
			\caption{Matrix illustration} \label{matrix}
		\end{center}
	\end{figure}
	\begin{remark}
		The estimate in Theorem \ref{refine} is sharp, since we have (\cite{My}, Proposition 5.6) $$\scl_{G_1*G_2}{[a,b]}=\frac{1}{2}-\frac{1}{\min(n_a,n_b)},$$ where $a\in G_1$, $b\in G_2$ and $n_a,n_b$ are the orders of $a,b$ respectively.
	\end{remark}
	Now we apply Theorem \ref{refine} with $N=+\infty$ to prove
	\setcounter{theoremx}{0}
	\begin{theoremx}
		Let $G=*_\lambda G_\lambda$ be a free product of torsion-free groups, and suppose $g\in [G,G]$ is not conjugate into any $G_\lambda$, then $$\scl_G(g)\ge1/2.$$
	\end{theoremx}
	\begin{proof}
		First notice that for each $g\in [G,G]$, there are finitely many $\lambda_1,\lambda_2,\ldots\lambda_n$ such that $g\in[H,H]$ where $H=*_{i=1}^n G_{\lambda_i}\le G$. Moreover, we have $\scl_H(g)=\scl_G(g)$ since $H$ is a retract of $G$ and $\scl$ is monotone under homomorphism. Thus it suffices to show $\scl_H(g)\ge1/2$.
		
		Now we induct on $n$. The case $n=2$ directly follows from Theorem \ref{refine}. Now suppose $n>2$, then $G$ is the free product of two torsion free groups $*_{i=1}^{n-1}G_{\lambda_i}$ and $G_{\lambda_n}$. If $g$ is not conjugate into either of them, then the result follows from Theorem \ref{refine}; otherwise, by assumption, $g$ is conjugate into $*_{i=1}^{n-1}G_{\lambda_i}$ but not any $G_{\lambda_i}$, then the result follows from the inductive assumption.
	\end{proof}
	Finally, Corollary \ref{clscl} and Theorem \ref{refine} together imply the following result about commutator length similar to the results obtained by Ivanov--Klyachko \cite{IK}:
	\begin{cor}\label{clversion}
		Let $G=A*B$ and $g=a_1b_1\cdots a_Lb_L$ with $a_i\in A\backslash\{id\}$, $b_i\in B\backslash\{id\}$ and $L\ge1$ such that $g\in[G,G]$. Let $N\ge2$ be the minimal order of $a_i$ and $b_i$, and $g_i$ be conjugates of $g$, then $$2\cdot\cl(g_1^{n_1}\ldots g_m^{n_m})-2\ge \sum_{i=1}^{m}(n_i-1)-\left[\frac{2}{N}\sum_{i=1}^{m} n_i\right],$$ where $[x]$ denotes the largest integer no greater than $x$.
	\end{cor}
	\bibliographystyle{amsplain}

\end{document}